\documentclass[psamsfonts]{amsart}

\usepackage{amssymb,amsfonts}
\usepackage[all,arc]{xy}
\usepackage{enumerate}
\usepackage{mathrsfs}
\usepackage{xcolor}
\usepackage[margin=1in]{geometry}
\usepackage{bbm}
\usepackage [english]{babel}
\usepackage [autostyle, english = american]{csquotes}
\MakeOuterQuote{"} 
\usepackage{xpatch}
\usepackage{hyperref} 
\makeatletter
\renewcommand\th@plain{\slshape}
\xpatchcmd{\proof}{\itshape}{\slshape}{}{}
\makeatother

\newcommand{\norm}[1]{\left\lVert#1\right\rVert}

\newcommand{\ds}{\displaystyle}
\newcommand{\R}{\mathbb{R}}
\newcommand{\Z}{\mathbb{Z}}

\newcommand{\card}{\mathrm{card}}

\newcommand\ssm{\smallsetminus}

\newcommand{\SL}{\mathrm{SL}_n(\R)}

\theoremstyle{plain}
\newtheorem{thm}{Theorem}[section]
\newtheorem{cor}[thm]{Corollary}
\newtheorem{rmk}[thm]{Remark}

\newtheorem{lem}[thm]{Lemma}

\theoremstyle{definition}
\newtheorem{defn}[thm]{Definition}

\newtheorem{notn}[thm]{Notation}

\theoremstyle{remark}

\numberwithin{equation}{section}

\title[Some Results on Random Unimodular Lattices]{Some Results on Random Unimodular Lattices}

\author{Mishel Skenderi}    
\address{
\begin{itemize}
\item[] Department of Mathematics
\item[] Brandeis University 
\item[] Waltham, MA $02454$\textendash$9110$
\item[] USA
\item[] \href{mailto:mskenderi@brandeis.edu}{{\tt mskenderi@brandeis.edu}} 
\end{itemize}}

\begin{document}

\begin{abstract}
Let $n \in \mathbb{Z}_{\geq 3}.$ Given any Borel subset $A$ of $\mathbb{R}^n$ with finite and nonzero measure, we prove that the probability that the set of primitive points of a random full-rank unimodular lattice in $\mathbb{R}^n$ does not contain any $\mathbb{R}$-linearly independent subset of $A$ of cardinality $(n-2)$ is bounded from above by a constant multiple, which depends only on $n$, of $\left(\mathrm{vol}(A)\right)^{-1}.$ This generalizes a result that is jointly due to J. S. Athreya and G. A. Margulis (see \cite[Theorem 2.2]{Log}). We also generalize independent results of C. A. Rogers (see \cite[Theorem 6]{MeanRog}) and W. M. Schmidt (see \cite[Theorem 1]{Metrical}) about primitive lattice points of random lattices to the case of primitive tuples of rank less than $\frac{n}{2}.$ In addition to the work of the authors who were just mentioned, a crucial element of this present paper is the usage of a rearrangement inequality due to Brascamp\textendash Lieb\textendash Luttinger (see \cite[Theorem 3.4]{BLL}). 
\end{abstract}  


\maketitle

\tableofcontents

\section{Introduction and Summary of Results}
In this paper, we let $n$ denote an arbitrary element of $\Z_{\geq 2}.$\footnote{We shall eventually make the more restrictive standing assumption that $n$ is an arbitrary element of $\Z_{\geq 3}.$ For the moment, however, the value $n=2$ is admissible.} We write $m$ for the Lebesgue measure on $\R^n.$ We let $X_n$ denote the space of all full-rank unimodular lattices in $\R^n,$ and we let $\eta$ denote the Haar probability measure on $X_n$. That is to say, we identify $X_n$ with $\SL/\mathrm{SL}_n(\Z)$ in the usual manner; we then let $\eta$ denote the Borel probability measure on $X_n$ that is induced in the canonical fashion by the Haar measure on $\SL.$ For the sake of clarity, we explicitly remark that the domain of $\eta$ is assumed to be the set of all Borel subsets of $X_n.$ For any $\Lambda \in X_n,$ we define \[ \Lambda_{\rm pr} := \left\lbrace v \in \Lambda : \ \text{there exists a} \ \text{$\Z$-basis} \  \mathcal{B} \ \text{of} \ \Lambda \ \text{for which} \ v \in \mathcal{B} \right\rbrace,\] the set of all \textsl{primitive} points of $\Lambda.$ It is well-known that for any $\Lambda \in X_n,$ the mapping \[ \Z_{\geq 1} \times \Lambda_{\rm pr} \to \Lambda \ssm \left\lbrace 0_{\R^n} \right\rbrace \ \text{given by} \ (k, v) \mapsto kv \] is a bijection. Lastly, we let $\zeta$ denote the Euler-Riemann zeta function.

\medskip

In a 1955 paper, C. A. Rogers proved the $n \geq 3$ case of the following theorem: see \cite[Theorem 6]{MeanRog}.\footnote{Rogers initially claimed to have proved the case $n=2$, as well; however, it was later discovered that there was an error in his proof in the case $n=2.$ The $n=2$ case is therefore due to Schmidt.} In a 1960 paper, W. M. Schmidt established certain quantitative results (see \cite[Theorems 1 and 2]{Metrical}) that yield the following theorem in its entirety: see \cite[Corollary]{Metrical}. 

\begin{thm}\label{Rogers-Schmidt} 
Let $A$ be any Borel subset of $\R^n$ with infinite measure. Then for almost every full-rank lattice $\mathcal{L}$ in $\R^n,$ the set $A \cap \mathcal{L}$ contains infinitely many primitive points of $\mathcal{L}.$ 
\end{thm}

In a 2006 paper, I. Aliev and P. M. Gruber generalized Theorem \ref{Rogers-Schmidt} as follows. 

\begin{thm}\label{AlievGruber}\cite[Theorem 1]{AG}
Let $A$ be any Borel subset of $\R^n$ with infinite measure. Then for almost every full-rank lattice $\mathcal{L}$ in $\R^n,$ the set $A^n \cap \mathcal{L}^n$ contains infinitely many pairwise disjoint $n$-tuples of linearly independent primitive points of $\mathcal{L}.$
\end{thm}

In a 2009 paper, J. S. Athreya and G. A. Margulis proved the following probabilistic analogue of the Minkowski Convex Body Theorem; their theorem may be regarded as a quantitative version of Theorem \ref{Rogers-Schmidt} above. 

\begin{thm}\label{origAM}\cite[Theorem 2.2]{Log} There exists a constant $C_n > 0$ such that if $A$ is any Borel subset of $\R^n$ with $0 < m(A) < +\infty,$ then 
\begin{equation}\label{AMeq}
\eta\left( \{\Lambda \in X_n : A \cap \Lambda = \varnothing \} \right) \leq \frac{C_n}{m(A)}. 
\end{equation}    
\end{thm} 

Throughout this paper, we make the usual convention that $\ds \mathrm{span}_{\R}\left(\varnothing\right) = \left\lbrace 0_{\R^n} \right\rbrace.$ Notice that the preceding theorem is then equivalent to the theorem obtained by replacing the inequality \eqref{AMeq} by the inequality 
\begin{equation}\label{modAMeq} 
    \eta\left( \{\Lambda \in X_n : \dim_\R\left( \mathrm{span}_{\R}\left( A \cap \Lambda \right)\right) < 1 \} \right) \leq \frac{C_n}{m(A)}. 
\end{equation}  

\medskip

\begin{rmk}\label{Standing} \rm 
Unless we explicitly state otherwise, we shall assume now and henceforth that $n$ is an arbitrary element of $\Z_{\geq 3}.$ 
\end{rmk}

\smallskip

Our first main result in this paper is the following theorem. It is a generalization of Theorem \ref{origAM}, in light of the fact that for any subset $A$ of $\R^n$ and any integers $j$ and $k$ with $1 \leq j \leq k,$ we have 
\begin{align*}
    \{\Lambda \in X_n : \ A \cap \Lambda = \varnothing \} &\subseteq  \{\Lambda \in X_n : \dim_\R\left( \mathrm{span}_{\R}\left( A \cap \Lambda \right)\right) < j \} \\
    &\subseteq  \{\Lambda \in X_n : \dim_\R\left( \mathrm{span}_{\R}\left( A \cap \Lambda_{\rm pr} \right)\right) < k \}.  
\end{align*}

\begin{thm}\label{Main}
There exists $\omega_n \in \R_{\geq 1}$ such that for any Borel subset $A$ of $\R^n$ with $0 < m(A) < +\infty,$ we have 
\begin{equation}\label{onemain}
    \eta\left( \left\lbrace\Lambda \in X_n : \dim_\R\left( \mathrm{span}_{\R}\left( A \cap \Lambda_{\rm pr} \right)\right) < n-2 \right\rbrace \right) \leq \frac{\omega_{n}}{m(A)}.
\end{equation} 
\end{thm}

\medskip

The following corollary is then a straightforward consequence of Theorem \ref{Main}; it is the analogue of \cite[Corollary 2.3]{Log}. 

\begin{cor}\label{InfMain}
Let $\ds \left( A_j \right)_{j \in \Z_{\geq 1}}$ be any sequence of Borel subsets of $\R^n$ with $\ds \lim_{j \to +\infty} m\left(A_j\right) = +\infty$ and such that for each $j \in \Z_{\geq 1},$ we have $0 < m\left(A_j\right) < +\infty.$ Then \[  \lim_{j \to +\infty} \eta\left(\left\lbrace \Lambda \in X_n : \dim_{\R}\left(\mathrm{span}_{\R}\left(\Lambda_{\rm pr} \cap A_j \right)\right) \geq n-2 \right\rbrace\right) = 1. \] Moreover, if $A$ is any Borel subset of $\R^n$ with $m(A) = +\infty,$ then \[ \eta\left(\left\lbrace \Lambda \in X_n : \dim_{\R}\left(\mathrm{span}_{\R}\left(\Lambda_{\rm pr} \cap A \right)\right) \geq n-2 \right\rbrace\right) = 1. \]  
\end{cor}

\medskip 

We also recover a weaker version of the result that was essentially proved by Aliev and Gruber in \cite[Theorem 1]{AG}. For the sake of clarity, let us first introduce some notation.   

\begin{defn}\label{countN}
If $A$ is any subset of $\R^n$ and $\Lambda \in X_n,$ then define \[ N_{n, n-2}(\Lambda, A) := \card\left( \left\lbrace \left( v_1, \dots , v_{n-2} \right) \in \left(\Lambda_{\rm pr}\right)^{n-2} \cap A^{n-2}: \dim_{\R}\left(\mathrm{span}_{\R}\left(\left\lbrace v_1, \dots , v_{n-2} \right\rbrace\right)\right) = n-2 \right\rbrace \right). \] 
\end{defn}

\begin{cor}\label{naught}
Let $A$ be any Borel subset of $\R^n$ with $m(A) = +\infty.$ Then \[ \eta\left( \left\lbrace \Lambda \in X_n : N_{n, n-2}(\Lambda, A) = \aleph_0 \right\rbrace\right) = 1. \] 
\end{cor}

\medskip

In order to state our second main result, it is necessary to make a new definition and establish some notation. 
\begin{defn}
\begin{itemize}
\item[\rm (i)] Let $k$ be any integer with $1 \leq k \leq n.$ Let $\Lambda \in X_n.$ Given any $(v_1, \dots , v_k) \in \Lambda^k,$ we say that $(v_1, \dots , v_k)$ is a \textsl{primitive $k$-tuple} if there exists a $\Z$-basis $\{w_1, \dots , w_n\}$ of $\Lambda$ with $\{v_1, \dots , v_k\} \subseteq \{w_1, \dots , w_n\}.$ We denote the set of all primitive $k$-tuples of $\Lambda$ by $\left(\Lambda^k\right)_{\rm pr}.$ 
\item[\rm (ii)]
Let $E$ be any subset of $\R^n,$ and let $k$ be any integer with $1 \leq k \leq n.$ Define $\mathrm{Pr}_{E, k} : X_n \to [0, +\infty]$ by $\ds \mathrm{Pr}_{E, k}(\Lambda) := \card\left( E^k \cap \left(\Lambda^k\right)_{\rm pr}\right).$ If $E$ is a Borel set, then notice that $\mathrm{Pr}_{E, k}$ is $\eta$-measurable. 
\end{itemize}
\end{defn}      

\begin{thm}\label{ThmInequality}
There exists $\omega_n'' \in \R_{> 0}$ such that for any Borel subset $A$ of $\R^n$ with $m(A) < +\infty$ and for any integer $\ell$ with $\ds 1 \leq \ell \leq \frac{n-1}{2},$ we have
\begin{equation}\label{RogType}
\norm{\mathrm{Pr}_{A, \ell} - \left(\int_{X_n} \mathrm{Pr}_{A, \ell} \ d\eta\right) \mathbbm{1}_{X_n}}_2 = \norm{\mathrm{Pr}_{A, \ell} - \left( \frac{\left(m(A)\right)^\ell}{\prod_{j=0}^{\ell-1} \zeta\left(n-j\right)} \right) \mathbbm{1}_{X_n}}_2 \leq \omega_n'' \left(m(A)\right)^{\frac{2\ell-1}{2}}.  
\end{equation}
\end{thm}  

As consequences of the above theorem, we have the following corollaries. Corollary \ref{CorMinkowski} is a conditional generalization of an obvious (and unstated) corollary of Theorem \ref{Main}: it is a generalization because for any integer $k$ with $1 \leq k \leq n$ and any subset $A$ of $\R^n,$ we have \[ \{\Lambda \in X_n : \dim_\R\left( \mathrm{span}_{\R}\left( A \cap \Lambda_{\rm pr} \right)\right) < k \} \subseteq \left\lbrace \Lambda \in X_n : A^k \cap \left(\Lambda^k\right)_{\rm pr} = \varnothing \right\rbrace. \] The second corollary is a conditional generalization of \cite[Theorem 1]{Metrical}, albeit without the error term. (However, it is quite likely possible to obtain an error term by arguing more carefully; see also \cite[Chapter 1, Lemma 10]{Sprind} for a result of V. G. Sprind\v zuk that drew upon and generalized the work of W. M. Schmidt.) In a sense, the second corollary is a quantitative analogue of Corollaries \ref{InfMain} and \ref{naught}. Let us also mention that the proof of Corollary \ref{CorSchmidt} is similar to an argument used in a probability textbook by R. T. Durrett: see \cite[Chapter 1, \S 6.8, Theorem]{Durrett}. The formulations and proofs of Theorem \ref{ThmInequality} and its corollaries have been motivated and informed by joint work of D. Ya. Kleinbock and the author: see the preprint \cite{KS}. 

\begin{cor}\label{CorMinkowski}
There exists a constant $\ds \omega_{n, *} \in \R_{>0}$ such that for any integer $\ell$ with $\ds 1 \leq \ell \leq \frac{n-1}{2}$ and any Borel subset $A$ of $\R^n$ with $0 < m(A) < +\infty,$ we have 
\begin{equation}\label{MinCorBound}
\eta\left( \left\lbrace \Lambda \in X_n : A^\ell \cap \left(\Lambda^\ell\right)_{\rm pr} = \varnothing \right\rbrace \right) \leq \frac{\omega_{n, *}}{m(A)}. \end{equation} \end{cor}

\begin{cor}\label{CorSchmidt}
Let $\ds \{A_t\}_{t \in \R_{\geq 1}}$ be a collection of Borel subsets of $\R^n$ that satisfies the following conditions:
\begin{itemize}
\item[\rm (i)] For any $r \in \R_{\geq 1},$ we have $0<m(A_r) < +\infty.$ 
\item[\rm (ii)] For any real numbers $r$ and $s$ with $1 \leq r \leq s < +\infty,$ we have $A_r \subseteq A_s.$
\item[\rm (iii)] The function $\R_{\geq 1} \to [m(A_1), +\infty)$ given by $t \mapsto m(A_t)$ is surjective.  
\end{itemize} 
Then for any integer $\ell$ with $\ds 1 \leq \ell \leq \frac{n-1}{2}$ and $\eta$-almost every $\Lambda \in X_n,$ we have \[ \lim_{t \to +\infty} \frac{\mathrm{Pr}_{A_t, \ell}(\Lambda)}{\left(m(A_t)\right)^\ell} = \frac{1}{ \prod_{j=0}^{\ell-1} \zeta\left(n-j\right)}. \] 
\end{cor}

\medskip

Naturally, we also obtain the following. 

\begin{cor}\label{tuplenaught}
For any integer $\ell$ with $\ds 1 \leq \ell \leq \frac{n-1}{2}$ and any Borel subset $A$ of $\R^n$ with $m(A) = +\infty,$ we have $\ds \eta\left(\left\lbrace \Lambda \in X_n : \mathrm{Pr}_{A, \ell}(\Lambda) = \aleph_0 \right\rbrace\right) = 1.$ 
\end{cor}

We also state and prove a generalization of Theorem \ref{origAM} that also generalizes Theorem \ref{Main} in the case $n=3.$ The author gratefully acknowledges that the statement and proof of the following theorem are due to D. Ya. Kleinbock. 

\begin{thm}[\cite{Kleinbock}]\label{DY}
Suppose $n$ is an arbitrary element of $\Z_{\geq 2}.$ There exists a constant $\ds \omega_n ' \in \R_{>0}$ such that for any Borel subset $A$ of $\R^n$ with  $0 < m(A) < +\infty,$ we have 
\begin{equation}\label{oneDY}
     \eta\left( \left\lbrace\Lambda \in X_n : \card\left( A \cap \Lambda_{\rm pr}\right) \leq 2 \right\rbrace \right) \leq \frac{\omega_n'}{m(A)}, 
\end{equation} whence \begin{equation}\label{primitiveDY}
\eta\left( \left\lbrace\Lambda \in X_n : \dim_\R\left( \mathrm{span}_{\R}\left( A \cap \Lambda_{\rm pr} \right)\right) < 2 \right\rbrace \right) \leq \frac{\omega_{n}'}{m(A)}. 
\end{equation}
\end{thm} 

\medskip

Let us also mention that Theorem \ref{DY} readily yields corollaries whose statements and proofs are analogous to those of Corollaries \ref{InfMain} and \ref{naught}. 

\begin{rmk} \rm
Suppose $n$ is an arbitrary element of $\Z_{\geq 2}.$ In the result \cite[Corollary 2.14]{Str}, A. Str\"ombergsson proves that the bound \eqref{AMeq} of Athreya and Margulis is sharp in the following precise sense: There exists some constant $\ds E_n \in \R_{>0}$ and there exists some $\ds V_n \in \R_{>0}$ such that for each $v \in \left[V_n, +\infty\right),$ there exists some bounded convex subset $A_v$ of $\R^n$ with $m\left(A_v\right) = v$ for which \begin{equation*}
     \eta\left( \{\Lambda \in X_n : \dim_\R\left( \mathrm{span}_{\R}\left( A_v \cap \Lambda \right)\right) < 1 \} \right) \geq \frac{E_n}{m\left(A_v\right)}. 
\end{equation*} It follows, \textit{a fortiori}, that the bounds in inequalities \eqref{onemain}, \eqref{MinCorBound}, \eqref{oneDY}, and \eqref{primitiveDY} are analogously sharp.           
\end{rmk}  

\subsection*{Acknowledgements} The author would like to thank D. Ya. Kleinbock, his doctoral adviser, for various discussions and comments appertaining to this paper. The author would like to thank J. S. Athreya and S. K. Fairchild for several helpful conversations regarding Rogers's higher moment formulae and related topics; in addition, the author would like to thank S. K. Fairchild for bringing to his attention the paper \cite{AG}. The author is immensely grateful to S. Kim for bringing to his attention the paper \cite{MeanSchmidt} and for remarking to him that what is essentially the statement of Lemma \ref{polball} in this paper follows from \cite[Theorem 2]{MeanSchmidt} and analysis of (21) in the proof of \cite[Theorem 2]{MeanSchmidt}. The author would like to thank L. L. Pham for a conversation that led the author to generalize results in an earlier version of this paper; he would also like to thank A. Gupta for an interesting discussion concerning the Euler-Riemann zeta function. Finally, the author would like to thank the anonymous referee for a meticulous reading and a thorough report whose suggestions helped improve this paper.  

\section{Proofs of Results}
The approach of Athreya and Margulis in \cite{Log} relies on the Siegel Mean Value Theorem \cite{Siegel} and its analogue for the second moment of the Siegel transform, which was proved by Rogers \cite[Theorem 4]{MeanRog}. Our approach here is a modification of that of Athreya and Margulis; we proceed by defining the Siegel transform and stating a version of Rogers's Theorem 4 that will be sufficient for our purposes in this note. 

\begin{defn}
If $f: \R^n \to \R_{\geq 0}$ is any function, then we define its \textsl{primitive Siegel transform} $\ds \widehat{f}^{{ \ }^{\rm pr}} : X_n \to [0, +\infty]$ and \textsl{Siegel transform} $\ds \widehat{f} : X_n \to [0, +\infty]$ by \[ \widehat{f}^{{ \ }^{\rm pr}}(\Lambda) := \sum_{v \in \Lambda_{\rm pr}} f(v) \hspace{0.7in} \text{and} \hspace{0.7in} \widehat{f}(\Lambda) := \sum_{v \in \left(\Lambda \ssm \left\lbrace 0_{\R^n} \right\rbrace \right)} f(v). \] Notice that if $f$ is Borel measurable, then each of $\ds \widehat{f}^{{ \ }^{\rm pr}}$ and $\ds \widehat{f}$ is $\eta$-measurable. 
\end{defn}

The celebrated Siegel Mean Value Theorem then states the following. 
\begin{thm}\label{SiegelMean}\cite{Siegel}
Let $\ds f: \R^n \to \R_{\geq 0}$ be any Borel measurable function in $\ds L^1\left(\R^n, m\right).$ Then \[ \int_{X_n} \widehat{f}^{{ \ }^{\rm pr}} \ d\eta = \left(\zeta(n)\right)^{-1}\int_{\R^n} f \ dm \hspace{0.7in} \text{and} \hspace{0.75in }\int_{X_n} \widehat{f} \ d\eta = \int_{\R^n} f \ dm. \] 
\end{thm}

\begin{rmk} \rm 
C. L. Siegel originally proved this Theorem under the assumptions that the codomain of $f$ is $\R$ and that $f$ is compactly supported, bounded, and Riemann-integrable. The version that we have stated above is a straightforward consequence. 
\end{rmk}

In order to state clearly Rogers's Theorem for higher moments of the Siegel transform, it is first necessary to develop some notation; our notation is the same as that in \cite{MeanRog, RogSet, MeanSchmidt}. 

\begin{defn}\label{rogdef}
Let $k$ be any integer with $2 \leq k \leq n-1.$ For any integer $r$ with $1 \leq r \leq k-1,$ we define $\mathcal{P}_{k, r}$ to be the set of all partitions $(\nu ; \mu) = (\nu_1, \dots , \nu_r ; \mu_1, \dots , \mu_{k-r})$ of the set $\{1, \dots , k\}$ subject to the following conditions:
\begin{itemize}
    \item $\ds 1 \leq \nu_1 < \dots < \nu_r \leq k$ \ and \ $\ds 1 \leq \mu_1 < \dots < \mu_{k-r} \leq k$; 
    \item for each $\ds i \in \{1, \dots , r\}$ and each $j \in \{1, \dots , k-r\},$ we have $\nu_i \neq \mu_j.$ 
\end{itemize}
For any integer $r$ with $1 \leq r \leq k-1$, any integer $s \in \Z_{\geq 1},$ and any $\ds (\nu ; \mu) = (\nu_1, \dots , \nu_r ; \mu_1, \dots , \mu_{k-r}) \in \mathcal{P}_{k, r},$ we define $\mathcal{D}_{k, r, s, (\nu ; \mu)}$ to be the set of $r \times k$ integer matrices $D = [d_{ij}]$ subject to the following conditions: 
\begin{itemize}
    \item each column of $D$ is nonzero;\footnote{\label{ambiguous} This condition is missing in both \cite{MeanRog} and \cite{MeanSchmidt}, but it is present in \cite{RogSet}. As we shall see, this ambiguity is immaterial to the results of this paper.}
    \item the greatest common factor of all coefficients of $D$ is relatively prime to $s$;
    \item if $i \in \{1, \dots , r\}$ and $j \in \{1, \dots , r\},$ then $\ds d_{i \nu_j} = s \delta_{i j}$; 
    \item if $i \in \{1, \dots , r \},$ $j \in \{1, \dots , k-r\},$ and $\mu_j < \nu_i$; then $\ds d_{i \mu_j} = 0.$  
\end{itemize}
For any integer $r$ with $1 \leq r \leq k-1$, any integer $s \in \Z_{\geq 1},$ any $\ds (\nu ; \mu) = (\nu_1, \dots , \nu_r ; \mu_1, \dots , \mu_{k-r}) \in \mathcal{P}_{k, r},$ and any $D = [d_{ij}] \in \mathcal{D}_{k, r, s, (\nu ; \mu)}$, we introduce the following notation: We let $\varepsilon_{D, 1}, \dots , \varepsilon_{D, r}$ denote the elementary divisors of the matrix $D$; for each $i \in \{1, \dots , r\},$ we then define $e_{D, i} := \gcd\left(\varepsilon_{D, i}, s\right).$ 
\end{defn}       

A special case of Rogers's Theorem then states the following. 
\begin{thm}\cite[Theorem 4]{MeanRog}\label{kmoment}
Let $\ds f: \R^n \to \R_{\geq 0}$ be any Borel measurable function. Let $k$ be any integer with $2 \leq k \leq n-1.$ Then 
\begin{align*}
    &\int_{X_n} \left( \widehat{f} \right)^k \ d\eta \\
    &= \left( \int_{\R^n} f \ dm \right)^k \\
    &+\sum_{r=1}^{k-1} \sum_{(\nu ; \mu) \in \mathcal{P}_{k, r}} \sum_{s \in \Z_{\geq 1}} \sum_{D \in \mathcal{D}_{k, r, s, (\nu ; \mu)}} \left( \frac{e_{D, 1} \cdots e_{D, r}}{s^r} \right)^n \underbrace{\int_{\R^n} \cdots \int_{\R^n}}_{r \ \textup{times}} \prod_{j=1}^k f\left(\sum_{i=1}^r \frac{d_{i j}}{s} x_i \right) \ dm(x_1) \dots dm(x_r). 
\end{align*}
Here, both sides of the equation may be equal to $+\infty.$ 
\end{thm} Rogers's approach to proving \cite[Theorem 4]{MeanRog} was quite different from Siegel's approach in proving the Siegel Mean Value Theorem. W. M. Schmidt then gave a proof of \cite[Theorem 4]{MeanRog} that proceeded along the lines of Siegel's proof: see \cite{German}. For a lucid and succint discussion of various results of this sort, see \cite[Chapter 1]{Kim}. \begin{rmk}\label{Schfin} \rm
Under the hypotheses of Theorem \ref{kmoment} and the additional hypotheses that $f: \R^n \to \R_{\geq 0}$ is bounded and compactly supported, W. M. Schmidt proved (\cite[Theorem 2]{MeanSchmidt}) that $\ds \int_{X_n} \left( \widehat{f} \right)^k \ d\eta < +\infty.$ In fact, even more is true: under the hypotheses of Theorem \ref{kmoment} and the aforementioned additional hypotheses, it follows from \cite[Lemma 3.1]{EMM} and \cite[Lemma 3.10]{EMM} that for each $p \in [1, n),$ we have $\ds \widehat{f} \in L^p(X_n).$ We note here that \cite[Lemma 3.1]{EMM} was proved by appealing to \cite[Lemma 2]{HeightSchmidt}.          
\end{rmk}

\begin{defn}
Let $t \in \R_{> 0}.$ We denote the closed Euclidean ball in $\R^n$ that is centered at the origin and has Lebesgue measure equal to $t$ by $B_t.$ We denote the indicator function of $B_t$ by $\rho_t.$
\end{defn}

\begin{rmk} \rm
Let $t \in \R_{>0}.$ Notice that for each $x \in \R^n,$ we have $\ds \rho_1\left(t^{-\frac{1}{n}} x\right) = \rho_t(x).$ 
\end{rmk} 

\begin{lem}\label{polball}
Let $k$ be any integer with $2 \leq k \leq n-1.$ The function $\beta_{n, k} :\R_{>0} \to \R_{>0}$ given by \[ \beta_{n, k}(t) := \int_{X_n} \left( \widehat{\rho_t} \right)^k \ d\eta \] is a real polynomial in $t$ that is monic and of degree $k.$ Each coefficient of this polynomial is nonnegative, and its constant term is zero. 
\end{lem}
\begin{proof}
We know that $\beta_{n, k}$ is well-defined. For each $r \in \{1, \dots , k-1\},$ define $\beta_{n, k, r} : \R_{>0} \to \R_{\geq 0}$ by \[ \beta_{n, k, r}(t) := \sum_{(\nu ; \mu) \in \mathcal{P}_{k, r}} \sum_{s \in \Z_{\geq 1}} \sum_{D \in \mathcal{D}_{k, r, s, (\nu ; \mu)}} \left( \frac{e_{D, 1} \cdots e_{D, r}}{s^r} \right)^n \underbrace{\int_{\R^n} \cdots \int_{\R^n}}_{r \ \textup{times}} \prod_{j=1}^k \rho_t\left(\sum_{i=1}^r \frac{d_{i j}}{s} x_i \right) \ dm(x_1) \dots dm(x_r). \] For any $t \in \R_{>0},$ any $r \in \{1, \dots , k-1\},$ any $(\nu ; \mu) \in \mathcal{P}_{k, r},$ any $s \in \Z_{\geq 1},$ and any $D \in \mathcal{D}_{k, r, s, (\nu ; \mu)};$ we have 
\begin{align*}
    &\underbrace{\int_{\R^n} \cdots \int_{\R^n}}_{r \ \textup{times}} \prod_{j=1}^k \rho_t\left(\sum_{i=1}^r \frac{d_{i j}}{s} x_i \right) \ dm(x_1) \dots dm(x_r) \\
    &=\underbrace{\int_{\R^n} \cdots \int_{\R^n}}_{r \ \textup{times}} \prod_{j=1}^k \rho_1\left(t^{- \frac{1}{n}} \sum_{i=1}^r \frac{d_{i j}}{s} x_i \right) \ dm(x_1) \dots dm(x_r) \\
     &=\underbrace{\int_{\R^n} \cdots \int_{\R^n}}_{r \ \textup{times}} \prod_{j=1}^k \rho_1\left(\sum_{i=1}^r \frac{d_{i j}}{s} \left(t^{- \frac{1}{n}}x_i\right) \right) \ dm(x_1) \dots dm(x_r) \\
     &= \underbrace{\int_{\R^n} \cdots \int_{\R^n}}_{r \ \textup{times}} \prod_{j=1}^k \rho_1\left(\sum_{i=1}^r \frac{d_{i j}}{s} z_i \right) \ (t dm(z_1)) \dots (t dm(z_r)) \\
     &= \left( \underbrace{\int_{\R^n} \cdots \int_{\R^n}}_{r \ \textup{times}} \prod_{j=1}^k \rho_1\left(\sum_{i=1}^r \frac{d_{i j}}{s} z_i \right) \ dm(z_1) \dots dm(z_r)\right) t^r. 
\end{align*}
It follows that for each $t \in \R_{>0},$ we have \[ \beta_{n, k}(t) = \int_{X_n} \left( \widehat{\rho_t} \right)^k \ d\eta = t^k + \sum_{r=1}^{k-1} \left( \beta_{n, k, r}(1) \cdot t^r \right). \] It is clear that each coefficient of this polynomial is nonnegative and that its constant term is zero.      
\end{proof}

\begin{defn}\label{polydefn}
Let $k$ be any integer with $2 \leq k \leq n-1.$ We let $P_{n, k}(T)$ denote the element of $\R[T]$ such that for each $t \in \R_{>0},$ we have $\ds P_{n, k}(t) = \int_{X_n} \left( \widehat{\rho_t} \right)^k \ d\eta.$ By the preceding lemma, this definition makes sense. We note that $P_{n, k}(T)$ is a monic polynomial of degree $k$, each of its coefficients is nonnegative, and its constant term is zero. We also define $Q_{n, k}(T) \in \R[T]$ by $\ds Q_{n, k}(T) := \left(\zeta(n)\right)^{-k} T^k - T^k + P_{n, k}(T).$ 
\end{defn}

\begin{lem}\label{nonzero}
Let $k$ be any integer with $2 \leq k \leq n-1.$ Let $\alpha$ denote the coefficient of the degree $(k-1)$ term of $P_{n, k}(T).$ Then $\alpha \geq 1.$
\end{lem}
\begin{proof}
We use the notation of Definition \ref{rogdef} and of Theorem \ref{kmoment}. Let $D = \left[ d_{ij} \right]$ be the $(k-1) \times k$ integer matrix defined as follows: For each $i \in \{1, \dots , k-1\}$ and each $j \in \{1, \dots , k-1\},$ we have $d_{ij} := \delta_{ij}.$ We have $d_{1, k} := 1$. If $k = 2,$ then this completes the definition of $D.$ If $k > 2,$ then we complete the definition of $D$ as follows: For each $\ds i \in \{2, \dots , k-1\},$ we have $d_{ik} := 0.$ Let $\ds (\nu ; \mu) = \left(\nu_1, \dots , \nu_{k-1} ; \mu_1\right) := \left(1, \dots , (k-1) ; k \right) \in \mathcal{P}_{k, (k-1)}.$ Note that each column of $D$ is nonzero and that $\ds D \in \mathcal{D}_{k, (k-1), 1, (\nu ; \mu)}.$\footnote{The possible pleonasm in the preceding sentence is due to the ambiguity in the definition of $\mathcal{D}_{k, (k-1), 1, (\nu ; \mu)}$ that was mentioned in Footnote \ref{ambiguous}.} It now follows from the proof of Lemma \ref{polball} that we have 
\begin{align*}
    \alpha &\geq \left( e_{D, 1} \cdots e_{D, k-1} \right)^n \underbrace{\int_{\R^n} \cdots \int_{\R^n}}_{(k-1) \ \textup{times}} \prod_{j=1}^k \rho_1\left(\sum_{i=1}^{k-1} d_{ij} x_i \right) \ dm(x_1) \dots dm(x_{k-1}) \\
    &\geq \underbrace{\int_{\R^n} \cdots \int_{\R^n}}_{(k-1) \ \textup{times}} \prod_{j=1}^k \rho_1\left(\sum_{i=1}^{k-1} d_{ij} x_i \right) \ dm(x_1) \dots dm(x_{k-1}) \\
    &= \underbrace{\int_{\R^n} \cdots \int_{\R^n}}_{(k-1) \ \textup{times}} \rho_1\left(x_1\right) \cdot \prod_{j=1}^{k-1} \rho_1\left(x_j\right) \ dm(x_1) \dots dm(x_{k-1}) \\
    &= \underbrace{\int_{\R^n} \cdots \int_{\R^n}}_{(k-1) \ \textup{times}} \prod_{j=1}^{k-1} \rho_1\left(x_j\right) \ dm(x_1) \dots dm(x_{k-1}) \\ 
    &= \left(m\left(B_1\right)\right)^{k-1} \\
    &= 1.  \end{align*} \end{proof}

\begin{rmk}\label{Schwarz} \rm 
We shall soon use the notion of a \textsl{Schwarz symmetrization} of a given Borel measurable function $\R^n \to \R_{\geq 0}$: see \cite[Definition 3.3]{BLL} for the definition of a Schwarz symmetrization. We note here that if $A$ is any Borel subset of $\R^n$ with $\ds 0 < m(A) < +\infty,$ then $\ds \rho_{m(A)}$ is a Schwarz symmetrization of $\ds \mathbbm{1}_A.$
\end{rmk}

We now make some definitions and prove some preliminary results. 
\begin{defn}\label{indepdefn}
Let $k$ be any integer with $1 \leq k \leq n.$ Define \[ Z_k := \left\lbrace (x_1, \dots , x_k) \in (\R^n)^k : \dim_{\R}\left( \mathrm{span}_{\R}\left( \{x_1, \dots , x_k \}\right)\right) = k \right\rbrace. \] Notice that $Z_k$ is a Borel subset of $(\R^n)^k.$

Given any Borel measurable $f : \R^n \to \R_{\geq 0},$ define $\ds ^{k, \, \rm{pr}}\widetilde{f} : X_n \to [0, +\infty]$ and $\ds ^{k}\widetilde{f} : X_n \to [0, +\infty]$ by \[ ^{k, \, \rm{pr}}\widetilde{f}(\Lambda) := \sum_{v \in \left(Z_k \cap \left(\Lambda_{\rm pr}\right)^k\right)} f(v_1) \cdots f(v_k) \hspace{0.55in} \text{and} \hspace{0.55in} ^{k}\widetilde{f}(\Lambda) := \sum_{v \in \left(Z_k \cap \Lambda^k\right)} f(v_1) \cdots f(v_k). \] Notice that if $f$ is Borel measurable, then each of \, $^{k, \, \rm{pr}}\widetilde{f}$ and $^{k}\widetilde{f}$ is $\eta$-measurable. 
\end{defn}

\medskip
The following theorem is standard, but the author was unable to find a reference in the literature for the first statement of the theorem. The second statement was first stated without proof by C. L. Siegel (see $2$) on page $347$ of \cite{Siegel}) and was proved by C. A. Rogers in \cite{MeanRog}: see the discussion concerning (8) on page 251 of \cite{MeanRog}. Notice also that it is fairly easy to deduce the second statement from the first statement: given any function $f$ as in Theorem \ref{indepintegral}, any $\Lambda \in X_n,$ and any $t \in \R_{>0},$ observe that \[ \sum_{v \in \left(\Lambda \ssm \left\lbrace 0_{\R^n} \right\rbrace \right)} f(v) = \sum_{k=1}^{+\infty} \sum_{v \in \Lambda_{\rm pr}} f(kv) \ \ \ \ \ {\rm and} \ \ \ \ \ \int_{\R^n} f(t x) \ dm(x) = t^{-n} \int_{\R^n} f(x) \ dm(x). \] The ergodic-theoretic argument that we shall use to prove the following theorem is quite different from those used by Siegel and Rogers; this type of argument has become standard ever since the groundbreaking paper \cite{Veech} of W. A. Veech situated the Siegel transform in the realm of ergodic theory. 

\smallskip

\begin{thm}\label{indepintegral}
Let $k$ be any integer with $1 \leq k \leq n-1.$ Let $f : \R^n \to \R_{\geq 0}$ be Borel measurable. Then \[ \int_{X_n} \ ^{k, \, \rm{pr}}\widetilde{f} \ d\eta = \left(\zeta(n)\right)^{-k} \left( \int_{\R^n} f \ dm \right)^k \hspace{0.5in} \text{and} \ \hspace{0.5in} \int_{X_n} \ ^{k}\widetilde{f} \ d\eta = \left( \int_{\R^n} f \ dm \right)^k. \] 
\end{thm}
\begin{proof}
In light of the remarks above, we give a proof only for the first statement. We regard $Z_k$ as a topological subspace of the Euclidean space $\left(\R^n\right)^k.$ Let $\iota : Z_k \hookrightarrow \left(\R^n\right)^k$ denote the inclusion map. Let $\ds m^{\otimes k}$ denote the Lebesgue measure on $\left(\R^n\right)^k$ after the latter has been restricted to the Borel $\sigma$-algebra of $\left(\R^n\right)^k.$ Note that $\ds m^{\otimes k}\left( (\R^n)^k \ssm Z_k \right) = 0.$ Let $\ds \left(m^{\otimes k}\right)_{| Z_k}$ denote the restriction of $\ds m^{\otimes k}$ to the Borel $\sigma$-algebra of $Z_k.$ Let $C_c\left(Z_k\right)$ denote the complex vector space of all continuous and compactly supported functions $Z_k \to \mathbb{C}.$             

Define $\ds \Psi : C_c\left(Z_k\right) \to \mathbb{C}$ by \[ \Psi(F) := \int_{X_n}\left( \sum_{v \in \left(Z_k \cap \left(\Lambda_{\rm pr}\right)^k\right)} F(v) \right) \ d\eta(\Lambda). \] For any $F \in C_c\left(Z_k\right)$, define $\ds M_F := \sup\left\lbrace |F(x)| : x \in Z_k \right\rbrace < +\infty$, and let $K_F$ be a compact subset of $\R^n$ with $\ds {\rm supp}(F) \subseteq \left(K_F\right)^k.$ Remark \ref{Schfin} then implies that for each $F \in C_c(Z_k),$ $\ds |\Psi(F)| \leq M_F \int_{X_n} \left( \widehat{\mathbbm{1}_{K_F}} \right)^k < +\infty.$ We thus conclude that $\Psi$ is well-defined. 

Notice that the $k$-diagonal action of $\mathrm{SL}_n(\R)$ on $Z_k$ is transitive (since $k \leq n-1$) and that the $\mathbb{C}$-linear functional $\Psi$ is invariant with respect to this action. Notice also that with respect to this same action, the Borel measure $\ds \left(m^{\otimes k}\right)_{| Z_k}$ is also obviously invariant. The foregoing observations, the version of the Riesz Representation Theorem \cite[Theorem 7.2]{Folland}, and the Radon-Nikodym Theorem now imply that there exists a Borel measure $\xi$ on $Z_k$ for which $\ds \Psi\left(-\right) = \int_{Z_k} \left(-\right) \ d\xi$ and that there exists a constant $c = c_{m, \xi} \in \R_{>0}$ for which  we have $\ds \xi = c\left(\xi + \left(m^{\otimes k}\right)_{| Z_k}\right).$ Clearly, $c \neq 1.$ We thus have $\ds \xi = \frac{c}{1-c} \left(m^{\otimes k}\right)_{| Z_k}.$ Set $\ds c' := \frac{c}{1-c}.$ Let $h : \R^n \to \R_{\geq 0}$ be a continuous and compactly supported function for which $\ds \int_{\R^n} h \ dm = 1.$ Define $H : \left(\R^n\right)^k \to \R_{\geq 0}$ by $\ds H(x_1, \dots , x_k) := \prod_{j=1}^k h\left(x_j\right).$ Since $\ds \int_{\R^n} H \ d \left(m^{\otimes k}\right) = \int_{Z_k} (H \circ \iota) \ d\left(m^{\otimes k}\right)_{| Z_k},$ Theorem \ref{SiegelMean} for the primitive Siegel transform now implies $\ds c' := \left(\zeta(n)\right)^{-k}.$ After performing some standard approximation arguments, we see that the proof is complete. 
\end{proof} 

\smallskip

The following lemma is of utmost importance. 
\begin{lem}\label{Majorize}
Let $k$ be any integer with $1 \leq k \leq n-1.$ Let $A$ be any Borel subset of $\R^n$ with $m(A) < +\infty.$ Then \[ \int_{X_n} \left(\widehat{\mathbbm{1}_A}^{{\rm pr}}\right)^k \ d\eta \leq Q_{n, k}\left(m(A)\right).  \] 
\end{lem}
\begin{proof}
This is an immediate consequence of Theorem \ref{kmoment}, Theorem \ref{indepintegral}, and \cite[Theorem 3.4]{BLL}. 
\end{proof}

\begin{rmk} \rm 
In his paper \cite{RogSing}, Rogers claims to prove an inequality that essentially constitutes \cite[Theorem 3.4]{BLL}: see \cite[Theorem 1]{RogSing}. However, Rogers's arguments in his proof of \cite[Theorem 1]{RogSing} are at times incomplete and elide certain details; thus, we have appealed to \cite[Theorem 3.4]{BLL}. 
\end{rmk}

To ease notation, let us make the following definitions. 

\begin{defn}
Let $k$ be any integer with $1 \leq k \leq n.$ Let $A$ be any Borel subset of $\R^n.$ We define \[ \Omega_{n, k}(A) := \left\lbrace \Lambda \in X_n : \ ^{k, \, \rm{pr}}\widetilde{\mathbbm{1}_A}(\Lambda) = 0 \right\rbrace. \] 

We also define $\ds \Theta_{n, k}(A) := X_n \ssm \Omega_{n, k}(A).$ Notice that \[ \Omega_{n, k}(A) = \left\lbrace \Lambda \in X_n : \dim_{\R}\left(\mathrm{span}_{\R}\left(A \cap \Lambda_{\rm pr} \right)\right) < k \right\rbrace \] and \[ \Theta_{n, k}(A) = \left\lbrace \Lambda \in X_n : \dim_{\R}\left(\mathrm{span}_{\R}\left(A \cap \Lambda_{\rm pr} \right)\right) \geq k \right\rbrace. \] 
\end{defn}

\begin{defn}
Define $\ds \Phi_{n} : \R_{>0} \to (0, 1)$ by \[ \Phi_{n}(t) := \left[\frac{\left(\zeta(n)\right)^{-(n-1)}t^{n-1}}{Q_{n, n-1}(t)} \right]^{n-2}. \] By Lemmata \ref{polball} and \ref{nonzero}, it follows $\ds \Phi_{n} : \R_{>0} \to (0, 1)$ is well-defined.
\end{defn}

We are now ready to prove our first main result. 

\begin{proof}[Proof of Theorem \ref{Main}] 
Let $\omega_n$ denote the sum of all the non-leading coefficients of $\ds \left[ \left(\zeta(n)\right)^{n-1}Q_{n, n-1}(T)\right]^{n-2}.$ By Lemma \ref{nonzero}, we have $\ds \omega_n \geq 1.$ For each $t \in \R_{\geq 1},$ we have 
\[
1 - \Phi_{n}(t) = \frac{\left[\frac{Q_{n, n-1}(t)}{\left(\zeta(n)\right)^{-(n-1)}t^{n-1}}\right]^{n-2} - 1}{\left[\frac{Q_{n, n-1}(t)}{\left(\zeta(n)\right)^{-(n-1)}t^{n-1}}\right]^{n-2}} \leq \left[\frac{Q_{n, n-1}(t)}{\left(\zeta(n)\right)^{-(n-1)}t^{n-1}}\right]^{n-2} - 1 \leq \omega_n t^{-1}. 
\] 

\medskip

Let $A$ be any Borel subset of $\R^n$ with $\ds 0 < m(A) < +\infty.$ Define $f : X_n \to [0, +\infty]$ by $\ds f := \ ^{(n-2), \, \rm{pr}}\widetilde{\mathbbm{1}_A}.$ Define $g : X_n \to \R_{\geq 0}$ by $\ds g := \mathbbm{1}_{\Theta_{n, n-2}(A)}.$ Set $\ds p := \frac{n-1}{n-2}.$ Note that $p > 1.$ Set $\ds q := \frac{p}{p-1} = n-1.$ Since $f = fg,$ it follows from H\" older's inequality and Theorem \ref{indepintegral} that we have 
\begin{align*}
   \|f\|_1 &\leq \|f\|_p \cdot \|g\|_q, \\
   \left(\zeta(n)\right)^{-(n-2)}(m(A))^{n-2} &\leq \|f\|_p \cdot \|g\|_q.
\end{align*} It follows $\ds \left(\zeta(n)\right)^{-(n-2)q}(m(A))^{(n-2) q} \leq \|f\|_p^q \cdot \eta\left(\Theta_{n, n-2}(A)\right).$ We have \[ 0 \leq f = {}^{(n-2), \, \rm{pr}}\widetilde{\mathbbm{1}_A} \leq \left(\widehat{\mathbbm{1}_A}^{{\rm pr}}\right)^{n-2},\] whence \[ 0 \leq f^p \leq \left(\widehat{\mathbbm{1}_A}^{{\rm pr}}\right)^{(n-2) p} = \left(\widehat{\mathbbm{1}_A}^{{\rm pr}}\right)^{n-1}. \] Lemma \ref{Majorize} now implies \[ \norm{f}_p^p \leq \int_{X_n} \left(\widehat{\mathbbm{1}_A}^{{\rm pr}}\right)^{n-1} \ d\eta \leq Q_{n, n-1}( m(A) ). \] Therefore, $\ds \|f\|_p^q \leq \left[ Q_{n, n-1}( m(A) ) \right]^\frac{q}{p}.$ It follows \[ \left(\zeta(n)\right)^{-(n-2)q}(m(A))^{(n-2) q} \leq \|f\|_p^q \cdot \eta\left(\Theta_{n, n-2}(A)\right) \leq \left[ Q_{n, n-1}( m(A) ) \right]^\frac{q}{p} \cdot \eta\left(\Theta_{n, n-2}(A)\right), \] whence 
\[ \Phi_n(m(A)) \leq \eta\left(\Theta_{n, n-2}(A)\right). \] The asserted inequality \eqref{onemain} now follows. (If $m(A) \leq 1,$ then inequality \eqref{onemain} is obviously true.) \end{proof}  

We now prove the corollaries of Theorem \ref{Main}. 

\begin{proof}[Proof of Corollary \ref{InfMain}] 
Letting $\omega_n$ be as in the statement of Theorem \ref{Main}, we have \[ 1 = \lim_{j \to +\infty} 1 - \omega_n\left(m\left(A_j\right)\right)^{-1} \leq \lim_{j \to +\infty} \eta\left(\Theta_{n, n-2}\left(A_j\right)\right).\] Let $A$ be any Borel subset of $\R^n$ with $m(A) = +\infty.$ For each $j \in \Z_{\geq 1},$ set $\ds F_j := A \cap B_j.$ By the foregoing, it follows $\ds 1 \leq \lim_{j \to +\infty} \eta\left(\Theta_{n, n-2}\left(F_j\right)\right) \leq \eta\left(\Theta_{n, n-2}\left(A\right)\right).$ \end{proof}

\begin{proof}[Proof of Corollary \ref{naught}]
Notice that for any Borel subset $S$ of $\R^n$ and any $\Lambda \in X_n,$ we have $\ds N_{n, n-2}(\Lambda, S) :=   {}^{(n-2), \, \rm{pr}}\widetilde{\mathbbm{1}_S}(\Lambda).$ For each $j \in \Z_{\geq 1},$ set $A_j := A \cap B_j.$ It follows from Corollary \ref{InfMain} that for each $j \in \Z_{\geq 1},$ we have $\ds \eta\left( \Omega_{n, n-2}\left( A \ssm A_j \right)\right) = 0.$ Since $\ds A = \bigcup_{j \in \Z_{\geq 1}} A_j,$ it follows  \[ \eta\left( \left\lbrace \Lambda \in X_n : N_{n, n-2}(\Lambda, A) < \aleph_0 \right\rbrace\right) \leq \eta\left( \bigcup_{j \in \Z_{\geq 1}} \Omega_{n, n-2}\left( A \ssm A_j \right) \right) = 0. \] Hence, $\ds \eta\left( \left\lbrace \Lambda \in X_n : N_{n, n-2}(\Lambda, A) = \aleph_0 \right\rbrace\right) = 1.$ \end{proof}  

In order to prove our second main result, we first establish some notation and recall a result of W. M. Schmidt. 

\begin{notn} For any integer $k$ with $1 \leq k \leq n-1,$ set $\ds \theta_{n, k} := \left(\prod_{j=0}^{k-1} \zeta(n-j)\right)^{-1}.$ This quantity clearly always belongs to $(0, 1).$ \end{notn} 

\begin{defn}\label{primdefn}
Let $k$ be any integer with $1 \leq k \leq n.$ Let $f: \R^n \to \R_{\geq 0}.$ Define $^{k}\overline{f} : X_n \to [0, +\infty]$ by \[ ^{k}\overline{f}(\Lambda) := \sum_{v \in \left(\Lambda^k\right)_{\rm pr}} f(v_1) \cdots f(v_k).\] Notice that if $f$ is Borel measurable, then $^{k}\overline{f}$ is $\eta$-measurable. 
\end{defn}

The following theorem is due to W. M. Schmidt: see \cite[Satz 14]{TheorySchmidt}. A  proof of the following theorem is also sketched in \cite[Theorem 7.3]{KM}. 

\begin{thm}\label{Satz}\cite[Satz 14]{TheorySchmidt}
Let $k$ be any integer with $1 \leq k \leq n-1.$ Let $f: \R^n \to \R_{\geq 0}$ be any Borel measurable function in $L^1(\R^n, m).$ Then \[ \int_{X_n} \ ^{k}\overline{f} \ d\eta = \theta_{n, k} \left( \int_{\R^n} f \ dm \right)^k. \] 
\end{thm}  

\medskip

\begin{proof}[Proof of Theorem \ref{ThmInequality}]
Clearly, $1 \leq \ell \leq n-1.$ Let $Z_{2\ell}$ be as per Definition \ref{indepdefn}. Let $E$ be any Borel subset of $\R^n$ with $m(E) < +\infty.$ It follows immediately from Theorem \ref{indepintegral} and Lemma \ref{Majorize} that we have 
\begin{align*}
\int_{X_n} \left( ^{\ell}\overline{\mathbbm{1}_E} \right)^2 \ d\eta\leq \left[ \int_{X_n}\left(\sum_{(v, w) \in \left( Z_{2\ell} \cap \left(\left(\Lambda^\ell\right)_{\rm pr}\right)^2\right)} \mathbbm{1}_{E^\ell}(v) \mathbbm{1}_{E^\ell}(w)\right) \ d\eta(\Lambda) \right] + P_{n, 2\ell}\left(m(E)\right) - \left(m(E)\right)^{2\ell}.   
\end{align*}
By Theorem \ref{Satz}, we have $\ds \int_{X_n} \ ^{\ell}\overline{\mathbbm{1}_E} = \theta_{n, \ell} \left(m(E)\right)^\ell.$ Thus, we need only show \begin{equation}\label{equality}
\int_{X_n}\left(\sum_{(v, w) \in \left( Z_{2\ell} \cap \left(\left(\Lambda^\ell\right)_{\rm pr}\right)^2\right)} \mathbbm{1}_{E^\ell}(v) \mathbbm{1}_{E^\ell}(w)\right) \ d\eta(\Lambda) = \left(\theta_{n, \ell} \left(m(E)\right)^\ell \right)^2.
\end{equation} By assumption, $2\ell \leq n-1$; the equation \eqref{equality} now follows by arguing as in the proof of Theorem \ref{indepintegral}, with Theorem \ref{Satz} now playing the role that was played by the primitive Siegel Mean Value Theorem in the previous proof. \end{proof}     

\begin{rmk} \rm It is easy to see that proofs similar to the one just given yield the following: \textsl{There exists ${\omega'''_n} \in \R_{> 0}$ such that for any Borel subset $A$ of $\R^n$ with $m(A) < +\infty$ and for any integer $\ell$ with $\ds 1 \leq \ell \leq \frac{n-1}{2},$ we have} 
\begin{equation}\label{bound1}
\norm{^{\ell, \, \rm{pr}}\widetilde{\mathbbm{1}_A} - \left(\int_{X_n} \ ^{\ell, \, \rm{pr}}\widetilde{\mathbbm{1}_A} \ d\eta\right) \mathbbm{1}_{X_n}}_2 = \norm{^{\ell, \, \rm{pr}}\widetilde{\mathbbm{1}_A} - \left( \frac{m(A)}{\zeta(n)}\right)^\ell \mathbbm{1}_{X_n}}_2 \leq {\omega'''_n} \left(m(A)\right)^{\frac{2\ell-1}{2}}  
\end{equation} \textsl{and}
\begin{equation}\label{bound2}
\norm{^{\ell}\widetilde{\mathbbm{1}_A} - \left(\int_{X_n} \ ^{\ell}\widetilde{\mathbbm{1}_A} \ d\eta\right) \mathbbm{1}_{X_n}}_2 = \norm{^{\ell}\widetilde{\mathbbm{1}_A} - \left( m(A) \right)^\ell \mathbbm{1}_{X_n}}_2 \leq {\omega'''_n} \left(m(A)\right)^{\frac{2\ell-1}{2}}. 
\end{equation} As will be seen from the proof of Corollary \ref{CorSchmidt}, \eqref{bound1} and \eqref{bound2} imply the obvious analogues of Corollary \ref{CorSchmidt}. 
\end{rmk}

\begin{proof}[Proof of Corollary \ref{CorMinkowski}] Let $\ell$ be any integer with $\ds 1 \leq \ell \leq \frac{n-1}{2}.$ Let $A$ be an arbitrary Borel subset of $\R^n$ with $\ds 0 < m(A) < +\infty.$ Let $\omega_n''$ be as in Theorem \ref{ThmInequality}. By Theorem \ref{Satz}, Markov's inequality, and Theorem \ref{ThmInequality}, we have 
\begin{align*}
\eta\left( \left\lbrace \Lambda \in X_n : A^\ell \cap \left(\Lambda^\ell\right)_{\rm pr} = \varnothing \right\rbrace\right) &\leq \eta\left( \left\lbrace \Lambda \in X_n : \left|\mathrm{Pr}_{A, \ell}(\Lambda) - \theta_{n, \ell} \left(m(A)\right)^\ell \right| \geq \theta_{n, \ell}\left(m(A)\right)^\ell \right\rbrace\right) \\
&\leq \norm{\mathrm{Pr}_{A, \ell} - \theta_{n, \ell} \left(m(A)\right)^\ell}_2^2 \left(\theta_{n, \ell}\left(m(A)\right)^\ell\right)^{-2} \\
&\leq \left(\omega_n''\right)^2 \left(m(A)\right)^{2\ell-1} \left(\theta_{n, \ell}\left(m(A)\right)^\ell\right)^{-2} \\
&= \left(\omega_n''\right)^2 \left(\theta_{n, \ell}\right)^{-2} \left(m(A)\right)^{-1} \\
&\leq \left(\omega_n''\right)^2 \left(\zeta(2)\right)^{2n} \left(m(A)\right)^{-1}. 
\end{align*} Now set $\ds \ds \omega_{n, *} := \left(\omega_n''\right)^2 \left(\zeta(2)\right)^{2n}.$ \end{proof}

\begin{proof}[Proof of Corollary \ref{CorSchmidt}] Let $\omega_n ''$ be as in Theorem \ref{ThmInequality}. Let $\ell$ be any integer with $\ds 1 \leq \ell \leq \frac{n-1}{2}.$ For each $t \in \R_{\geq 1},$ set $\ds h_t := \mathrm{Pr}_{A_t, \ell}.$ For each $t \in \R_{\geq 1},$ Theorem \ref{Satz} implies $\ds \int_{X_n} h_t \ d\eta = \theta_{n, \ell} \left(m(A_t)\right)^\ell.$ Fix $M \in \Z_{\geq 1}$ with $\ds M > m\left(A_1\right).$ For each $k \in \Z_{\geq M},$ fix $t_k \in \R_{\geq 1}$ for which $m\left(A_{t_k}\right) = k^2.$ Let $\varepsilon \in \R_{>0}$ be given. For each $k \in \Z_{\geq M},$ Markov's inequality and Theorem \ref{ThmInequality} imply 

\begin{align*}
\mu_X\left(\left\lbrace \Lambda \in X_n : \left| \frac{h_{t_k}(\Lambda)}{\theta_{n, \ell} \left(m\left(A_{t_k}\right)\right)^\ell} - 1 \right| \geq \varepsilon \right\rbrace\right) &\leq \varepsilon^{-2} \left\|\frac{h_{t_k}}{\theta_{n, \ell} \left(m\left(A_{t_k}\right)\right)^\ell} - \mathbbm{1}_{X}\right\|_2^2 \leq \varepsilon^{-2} \left(\omega_n''\right)^2 \left(\theta_{n, \ell}\right)^{-2} \left(m\left(A_{t_k}\right)\right)^{-1}.
\end{align*} Since $\varepsilon \in \R_{>0}$ is arbitrary and $\ds \sum_{k=M}^{+\infty} \left(m\left(A_{t_k}\right)\right)^{-1} = \sum_{k=M}^{+\infty} k^{-2} < +\infty,$ the Borel-Cantelli Lemma now implies that for $\eta$-almost every $\Lambda \in X_n,$ we have \[ \lim_{k \to +\infty} \frac{h_{t_k}(\Lambda)}{\left(m\left(A_{t_k}\right)\right)^\ell} = \theta_{n, \ell}. \] For any $k \in \Z_{\geq M}$ and any $\ds t \in \left[t_k, t_{k+1}\right),$ we have

\begin{align*}
\frac{k^{2\ell}}{(k+1)^{2\ell}} \frac{h_{t_k}}{\left(m\left(A_{t_k}\right)\right)^\ell} = \frac{h_{t_k}}{\left(m\left(A_{t_{k+1}}\right)\right)^\ell} \leq \frac{h_{t}}{\left(m\left(A_{t}\right)\right)^\ell} \leq \frac{h_{t_{k+1}}}{\left(m\left(A_{t_k}\right)\right)^\ell} = \frac{(k+1)^{2\ell}}{k^{2\ell}}\frac{h_{t_{k+1}}}{\left(m\left(A_{t_{k+1}}\right)\right)^\ell}. 
\end{align*} The result follows. \end{proof}

\begin{proof}[Proof of Corollary \ref{tuplenaught}]
Let $\ell$ be any integer with $\ds 1 \leq \ell \leq \frac{n-1}{2}.$ Let $A$ be any Borel subset of $\R^n$ with $m(A) = +\infty.$ There exists $K \in \Z_{\geq 1}$ with $m\left(A \cap B_K\right) > 0.$ It now follows from the inequality \eqref{RogType} and a basic fact of real analysis that there exists a strictly increasing sequence of integers $\left(t_k\right)_{k \in \Z_{\geq 1}}$ for which $t_1 > K$ and for which the following holds: For $\eta$-almost every $\Lambda \in X_n,$ we have \[ \lim_{k \to +\infty} \frac{\mathrm{Pr}_{\left(A \cap B_{t_k}\right), \ell}(\Lambda)}{\left(m\left(A \cap B_{t_k}\right)\right)^\ell} = \theta_{n, \ell}. \] The result follows. 
\end{proof}  

\begin{proof}[Proof of Theorem \ref{DY}] We follow the explanation that was given in \cite{Kleinbock}. Recall the assumption that $n$ is an arbitrary element of $\Z_{\geq 2}.$ Notice that the bound \eqref{primitiveDY}  is an immediate consequence of the bound \eqref{oneDY}, which we now proceed to prove. If $E$ is any Borel subset of $\R^n,$ then set $\ds P_E := \widehat{\mathbbm{1}_E}^{{\rm pr}},$ so as to ease notation. 

\medskip

It follows from \cite[(4.4)]{Log} in the case $n=2$ and from \cite[(0.2)]{KY} (for instance) in the case $n \geq 3$ that there exists a constant $s_n \in \R_{>0}$ such that for any Borel subset $E$ of $\R^n$ with $\ds 3\zeta(n) < m(E) < +\infty$ and with $\ds m(E)$ sufficiently large, we have 
\begin{align*}
\eta\left( P_E^{-1}\left( [0, 2] \right) \right) \cdot \left[ \frac{m(E)}{\zeta(n)} - 2 \right]^2 \leq \norm{P_E - \frac{m(E)}{\zeta(n)}\mathbbm{1}_{X_n}}_2^2 \leq s_n \, m(E), 
\end{align*} whence $\ds \eta\left( P_E^{-1}\left( [0, 2] \right) \right) \leq s_n \, m(E) \cdot \left[ \frac{m(E)}{\zeta(n)} - 2 \right]^{-2}.$ By choosing $\omega_n' \in \R_{>0}$ to be sufficiently large, the desired bound \eqref{oneDY} follows. \end{proof}  

\begin{rmk} \rm 
Suppose $n$ is an arbitrary element of $\Z_{\geq 2}.$ Notice that obtaining lower bounds for the measures of sets as in inequalities \eqref{onemain} and \eqref{MinCorBound} (or, equivalently, upper bounds for the measures of their complements in $X_n$) is an easy task: for any integer $k$ with $1 \leq k \leq n-1$ and any Borel subset $A$ of $\R^n$ with $m(A) < +\infty,$ it is an immediate consequence of Theorems \ref{indepintegral} and \ref{Satz} that we have
\begin{align*}
    \eta\left(\left\lbrace \Lambda \in X_n : \dim_{\R}\left(\mathrm{span}_{\R}\left(A \cap \Lambda \right)\right) \geq k \right\rbrace\right) &\leq \min\left\lbrace 1, \frac{(m(A))^k}{k!}\right\rbrace, \\
    \eta\left(\left\lbrace \Lambda \in X_n : \dim_{\R}\left(\mathrm{span}_{\R}\left(A \cap \Lambda_{\rm pr} \right)\right) \geq k \right\rbrace\right) &\leq \min\left\lbrace 1, \frac{(\zeta(n))^{-k} (m(A))^k}{k!}\right\rbrace, \  {\rm  and} \\
    \eta\left( \left\lbrace \Lambda \in X_n : A^k \cap \left(\Lambda^k \right)_{\rm pr} \neq \varnothing \right\rbrace\right) &\leq \min\left\lbrace 1, \frac{\theta_{n, k} (m(A))^k}{k!} \right\rbrace. 
\end{align*}
\end{rmk}

\end{document}